\documentclass[11pt,reqno]{amsart}
\usepackage[margin=1in]{geometry}

\usepackage[dvipsnames]{xcolor}
\usepackage{amssymb,amsmath,amsthm}
\usepackage{mathtools}
\usepackage{graphicx}
\usepackage{subcaption}
\usepackage{fullpage}
\usepackage{scrextend}
\usepackage{siunitx}

\usepackage{enumerate}
\usepackage{tikz,calc}
\usepackage{tikz-cd}
\usetikzlibrary{automata,positioning}
\usepackage{comment}
\usetikzlibrary{calc}
\usepackage{epsfig,graphicx}
\usepackage{listings}
\usepackage{enumitem}
\usepackage{microtype}

\usepackage{float}
\usepackage{bbm}
\usepackage[colorlinks=true, pdfstartview=FitH, linkcolor=NavyBlue, citecolor=RawSienna, urlcolor=RoyalBlue]{hyperref}

\newcommand{\custMR}[1]{\href{http://www.ams.org/mathscinet-getitem?mr=#1}{MR#1}}
\newcommand{\arxiv}[1]{\href{http://arxiv.org/abs/#1}{arXiv:#1}}

\newcommand{\brax}[1]{\left( #1 \right)}

\newcommand{\set}[1]{\left\{#1\right\}}

\renewcommand{\leq}{\leqslant}
\renewcommand{\le}{\leqslant}
\renewcommand{\geq}{\geqslant}

\newcommand{\N}{\mathbb{N}}

\newcommand{\ex}{\mathrm{ex}}

\newcommand{\Mod}[1]{\ (\text{mod}\ #1)}
\renewcommand{\Mod}[1]{{\ifmmode\text{\rm\ (mod~$#1$)}\else\discretionary{}{}{\hbox{ }}\rm(mod~$#1$)\fi}}

\newcommand{\allnotes}[1]{}
\renewcommand{\allnotes}[1]{\textit{#1}}

\newtheorem{theorem}{Theorem}[section]

\theoremstyle{definition} 

\numberwithin{theorem}{section}

\allowdisplaybreaks

\begin{document}

\title{The maximum number of triangles in $K_{1,s,t}$-free graphs}

\author{Asier Calbet}
\address{Department of Mathematics and Mathematical Statistics \\ Umeå University \\ Umeå \\ Sweden}
\email{asier.calbet@umu.se}

\author{Ritesh Goenka}
\address{Mathematical Institue \\ University of Oxford \\ Oxford, UK}
\email{ritesh.goenka@maths.ox.ac.uk}

\subjclass[2020]{05C35}
\keywords{Generalized Tur\'an problems, wheel-free graphs, triangle removal lemma.}

\begin{abstract}
    We consider the following generalized Tur\'an problem: For $2 \le s \le t$, what is the maximum number of triangles in a $K_{1,s,t}$-free graph on $n$ vertices? The previously best known lower and upper bounds are $\Omega(n^2)$ and $o(n^{3-1/s})$, respectively. To the best of our knowledge, all known proofs of the upper bound use the triangle removal lemma. We give a new elementary proof that avoids the use of the triangle removal lemma and improves the upper bound to $O\brax{n^{3-1/s}(\log n)^{-1+1/s}}$.
\end{abstract}

\maketitle

\section{Introduction}
\label{sec:intro}

Tur\'an problems are among the oldest and most well-studied questions in extremal combinatorics. In a seminal paper, Alon and Shikhelman~\cite{AS} initiated the study of the \emph{generalized Tur\'an number} $\ex(n, T, H)$, which is defined as the maximum number of copies of $T$ in an $H$-free graph on $n$ vertices. The classical Tur\'an numbers correspond to the case when $T$ is an edge. Several results about generalized Tur\'an numbers were already known long before Alon and Shikhelman~\cite{AS} defined them formally. See \cite{GP} for such results and a comprehensive survey of generalized Tur\'an problems.

A simple and natural case to consider is $T = K_3$. Alon and Shikhelman  determined $\ex(n, K_3, H)$ asymptotically when $\chi(H) \geq 4$~\cite[Proposition~2.2]{AS}. They also obtained several estimates when $\chi(H) = 2$ (see \cite[Section~4]{AS}). We consider the case when $\chi(H) = 3$. Mubayi and Mukherjee~\cite{MM} considered the case when $H$ is the suspension of a bipartite graph, meaning $H$ can be obtained by adding a new vertex incident to all vertices in a bipartite graph. Among other results, they showed that $\ex(n, K_3, K_{1,s,t})= o(n^{3-1/s})$ for fixed $1 \le s \le t$ \cite[Theorem~1.1]{MM}. Moreover, they proved $\ex(n, K_3, \widehat{C}_{2k})= o(n^{2+1/k})$ for fixed $k \ge 2$, where $\widehat{H}$ denotes the suspension of a graph $H$ \cite[Theorem~1.2]{MM}. Generalizing their results, Methuku~\cite[Theorem~6.1]{MM} showed that for any bipartite graph $H$ with $\ex(n,H) = O(n^\alpha)$ for some $1 < \alpha < 2$, we have $\ex(n, K_3, \widehat{H}) = o(n^{1+\alpha})$. The proofs use the triangle removal lemma \cite{RS}. Note that the bound $\ex(n, K_3, \widehat{H}) = O(n^{1+\alpha})$ is trivial. Indeed, in an $\widehat{H}$-free graph, the neighborhood of any vertex is $H$-free, so contains $O(n^{\alpha})$ edges. Equivalently, every vertex is in  $O(n^{\alpha})$ triangles, so summing over all vertices gives the result.

The first result of Mubayi and Mukherjee~\cite{MM} was generalized by Balogh, Jiang, and Luo~\cite{BJL}. They showed that $\ex(n, K_m, K_{s_1, s_2, \dots, s_m}) = o\brax{n^{m-1/\prod_{i=1}^{m-1} s_i}}$ for fixed positive integers $m \ge 3$ and $s_1 \le s_2 \le \dots \le s_m$ \cite[Theorem~1]{BJL}. In particular, their results imply $\ex(n, K_3, K_{r,s,t}) = o(n^{3-1/(rs)})$ for fixed $1 \le r \le s \le t$. Their proof uses the graph removal lemma. Basu, R\"{o}dl, and Zhao~\cite[Theorem~2]{BRZ} further generalized their result to hypergraphs using the hypergraph removal lemma. Note that the bound $\ex(n, K_m, K_{s_1, s_2, \dots, s_m}) = O\brax{n^{m-1/\prod_{i=1}^{m-1} s_i}}$ is again trivial. Indeed, given a $K_{s_1, s_2, \dots, s_m}$-free graph, one can construct a $K_{s_1, s_2, \dots, s_m}^{(m)}$-free $m$-uniform hypergraph (where $K_{s_1, s_2, \dots, s_m}^{(m)}$ is the complete $m$-uniform $m$-partite hypergraph with parts of size $s_1$,$s_2$,$\dots$,$s_m$) on the same vertex set by taking the edges to be the copies of $K_m$, so by the standard bound on the hypergraph Tur\'an number~\cite{Erd2}, we have $\ex(n, K_m, K_{s_1, s_2, \dots, s_m}) \leq \ex(n, K_{s_1, s_2, \dots, s_m}^{(m)}) = O\brax{n^{m-1/\prod_{i=1}^{m-1} s_i}}$.

We now restrict our attention to $\ex(n, K_3, K_{1,s,t})$ for fixed $1 \le s \le t$. The case $s = 1$ and $t = 2$ is easily seen to be equivalent to the $(6, 3)$-problem of Ruzsa and Szemer\'edi~\cite{RS}. As stated in \cite{BJL}, the best known bounds in this case are
\begin{equation*}
    n^2 e^{-O(\sqrt{\log n})} \le \ex(n, K_3, K_{1,1,2}) \le o(n^2),
\end{equation*}
where the lower bound comes from Behrend's construction~\cite{Beh} and the upper bound follows from the triangle removal lemma. For $t \geq 3$, we still have the same lower and upper bounds for $\ex(n, K_3, K_{1,1,t})$. In the case $s \ge 2$, we prove the following theorem, which improves upon the previously best known upper bound for $\ex(n, K_3, K_{1,s,t})$.

\begin{theorem}
\label{thm:main}
    Let $s, t \in \N$ with $2 \le s \le t$. Then $\ex(n,K_3,K_{1,s,t}) = O\brax{n^{3-1/s}(\log n)^{-1+1/s}}$, where the implicit constant depends on $s$ and $t$.    
\end{theorem}

In their abstract, Mubayi and Mukherjee~\cite{MM} state that perhaps stronger results than theirs can be proved without using the triangle removal lemma, and indeed our proof of Theorem~\ref{thm:main} does not use the triangle removal lemma. The proof of the previously best known upper bound for $\ex(n, K_3, K_{1,s,t})$ does use the triangle removal lemma, which yields a bound of $O\brax{n^{3-1/s} e^{-\Omega(\log^* n)}}$, where $\log^*$ is the iterated logarithm function (see \cite{Fox}). Theorem~\ref{thm:main} provides a polylogarithmic improvement to this bound.

We mention the special case $\ex(n, K_3, K_{1,2,2})$ for which the previously best known bounds were
\begin{equation}
\label{eqn:K122}
    \Omega(n^2) \le \ex(n, K_3, K_{1,2,2}) \le o(n^{5/2}),
\end{equation}
where the lower bound is obtained by taking a complete bipartite graph with the sizes of the parts as equal as possible and adding a matching inside one of the parts (see problem 2 in \cite{Erd} for a related conjecture of Gallai, which was later disproved by F\"uredi, Goemans, and Kleitman~\cite{FGK}) and the upper bound follows from the result of Mubayi and Mukherjee~\cite[Theorem~1.1]{MM}. The upper bound was also proved independently by Methuku, Gr\'osz, and Tompkins in an unpublished work (see \cite{MM}) and Balogh (see his lecture note \cite{Bal}). Reducing the gap between the bounds in \eqref{eqn:K122} has been raised as an open problem in several previous works~\cite{BJL,MM,MV}. Theorem~\ref{thm:main} provides a polylogarithmic improvement to the upper bound; however, the polynomial gap remains.

\section{Proof}
\label{sec:proof}

We begin by recalling the K\H{o}v\'ari-S\'os-Tur\'an theorem~\cite{KST}, which shall be useful in proving our main result. Let $z(m, n; s, t)$ denote the maximum number of edges in a $K_{s,t}$-free bipartite graph $G$ with $m$ vertices on one side and $n$ vertices on the other side. The Zarankiewicz problem~\cite{Zar} asks for an estimate on $z(m, n; s, t)$. K\H{o}v\'ari, S\'os, and Tur\'an~\cite{KST} proved the following theorem, which provides an upper bound on $z(m, n; s, t)$.

\begin{theorem}[K\H{o}v\'ari-S\'os-Tur\'an]
\label{thm:KST}
   For positive integers $s \le t$, we have $z(m, n; s, t) \le K(n^{1-1/s} m + n)$, where $K > 0$ is a constant depending only on $s$ and $t$.
\end{theorem}

\begin{proof}[Proof of Theorem~\ref{thm:main}]
    Let $G$ be a $K_{1,s,t}$-free graph on $n$ vertices. We need to show that $G$ has $O\brax{n^{3-1/s} (\log n)^{-1+1/s}}$ triangles. It will be more convenient to count ordered triangles, that is, ordered triples $(a,b,c)$ of vertices where $a$ is adjacent to $b$, $b$ is adjacent to $c$, and $c$ is adjacent to $a$. Let $V$ be the vertex set of $G$ and for subsets $A,B,C \subseteq V$, let $\Delta(A,B,C)$ be the number of ordered triangles $(a,b,c)$ with $a \in A$, $b \in B$ and $c \in C$. Then $\Delta(V,V,V)$ is six times the number of triangles in $G$, so we need to show that $\Delta(V,V,V) = O\brax{n^{3-1/s} (\log n)^{-1+1/s}}$.

    Let $\mathcal{P}$ be a partition of $V$ into $O(n/\log n)$ parts with $|P| \leq \log n / \log 8$ for each $P \in \mathcal{P}$. Then $\Delta(V,V,V) = \sum_{P \in \mathcal{P}} \Delta(P,V,V)$, so it suffices to show that $\Delta(P,V,V) = O(n^{2-1/s} (\log n)^{1/s})$ for each $P \in \mathcal{P}$. Let $P$ be such a part and for each subset $S \subseteq P$, let $V_S = \set{v \in V \,\colon\, \Gamma(v) \cap P = S}$, where $\Gamma(v)$ denotes the neighborhood of the vertex $v$. Note that the sets $V_S$ partition $V$. We will show that $\Delta(P, V_S, V) = O\brax{n^{1-1/s} (\log n)^{1/s} |V_S| + n \log n}$ for every subset $S \subseteq P$. It then follows that
    \begin{align*}
        \Delta(P,V,V) &= \sum_{S \subseteq P} \Delta(P, V_S, V)\\
        &= \sum_{S \subseteq P} O\brax{n^{1-1/s} (\log n)^{1/s} |V_S| + n \log n}\\
        &= O\brax{n^{2-1/s} (\log n)^{1/s} + 2^{|P|} \ n \log n},
    \end{align*}
    and hence $\Delta(P,V,V) = O(n^{2-1/s} (\log n)^{1/s})$, as desired, since $2^{|P|} \leq n^{1/3}$.

    Let $S \subseteq P$. We need to show that $\Delta(P,V_S,V) = O\brax{n^{1-1/s} (\log n)^{1/s} |V_S| + n \log n}$. Note that $\Delta(P,V_S,V) = \Delta(S,V_S,V)$ and that the induced bipartite subgraph between $S$ and $V_S$ is complete. Hence, either $|\Gamma(v) \cap S| < s$ or $|\Gamma(v) \cap V_S| < t$ for every vertex $v \in V$, for otherwise $G$ would contain a copy of $K_{1,s,t}$. Equivalently, we have $V = A \cup B$, where $A = \set{v \in V \,\colon\, |\Gamma(v) \cap S| < s}$ and $B = \set{v \in V \,\colon\, |\Gamma(v) \cap V_S| < t}$, so $\Delta(S,V_S,V) \leq \Delta(S,V_S,A) + \Delta(S,V_S,B)$. Since each vertex in $B$ is adjacent to at most $t$ vertices in $V_S$, we have $\Delta(S,V_S,B) \leq t|S| |B| = O(n \log n)$, so it remains to show that $\Delta(S,V_S,A) = O\brax{ n^{1-1/s} (\log n)^{1/s} |V_S| + n \log n}$.

    For each vertex $v \in S$, let $V_S'$ be a copy of $V_S$ and let $(\Gamma(v) \cap A)'$ be a copy of $\Gamma(v) \cap A$, with $V_S'$ and $(\Gamma(v) \cap A)'$ disjoint. Let $G_v$ be the bipartite graph with parts $V_S'$ and $(\Gamma(v) \cap A)'$ where the copy $u' \in V_S'$ of a vertex $u \in V_S$ is adjacent to the copy $w' \in (\Gamma(v) \cap A)'$ of a vertex $w \in \Gamma(v) \cap A$ if and only if $u$ is adjacent to $w$ in $G$. Then $G_v$ is $K_{s,t}$-free since $G$ is $K_{1,s,t}$-free, so by the K\H{o}v\'ari-S\'os- Tur\'an theorem (Theorem~\ref{thm:KST}), we have $e(G_v) = O\brax{|\Gamma(v) \cap A|^{1-1/s} |V_S| + |\Gamma(v) \cap A|}$, where $e(G_v)$ denotes the number of edges in $G_v$. Further, we have
    \begin{equation*}
        \sum_{v \in S} |\Gamma(v) \cap A| = \sum_{v \in A} |\Gamma(v) \cap S| \leq s |A| = O(n) .
    \end{equation*}
    Hence, by Jensen's Inequality for the concave function $x \mapsto x^{1-1/s}$, we have
    \begin{align*}
        \Delta(S,V_S,A) &= \sum_{v \in S} e(G_v)\\
        &= \sum_{v \in S} O\brax{|\Gamma(v) \cap A|^{1-1/s} |V_S| + |\Gamma(v) \cap A|}\\
        &= O\brax{n^{1-1/s} |S|^{1/s} |V_S| + n}\\
        &= O\brax{n^{1-1/s} (\log n)^{1/s} |V_S| + n}, 
    \end{align*}
    as required.
\end{proof}

\section*{Acknowledgments}

We are grateful to Yaobin Chen and Ji Zeng for several helpful discussions in the initial stages of the project. We also thank J\'ozsef Balogh and Peter Keevash for helpful conversations. The research leading to this note began during the Early Career Researchers in Combinatorics (ECRiC) workshop, 15-19 July 2024, funded by the International Centre for Mathematical Sciences, Edinburgh. The first author is supported by a Kempe Foundation scholarship. The second author is supported by a joint Clarendon and Exeter College SKP scholarship.

\end{document}